\documentclass{amsart}

\RequirePackage[T1]{fontenc}
\RequirePackage[all]{xy}
\RequirePackage{algorithm}
\RequirePackage{amscd}
\RequirePackage{amsfonts}
\RequirePackage{amsmath}
\RequirePackage{amssymb}
\RequirePackage{graphicx}
\RequirePackage{epstopdf}
\RequirePackage{euscript}
\RequirePackage{fancyhdr}
\RequirePackage{latexsym}
\RequirePackage{verbatim}
\RequirePackage{mathrsfs}
\RequirePackage{ifthen}
\RequirePackage{stmaryrd}

\renewcommand{\geq}{\geqslant}
\renewcommand{\leq}{\leqslant}

\newcommand{\ib}{\item[{\rm (i)}]}
\newcommand{\ii}{\item[{\rm (ii)}]}

\usepackage{a4wide}
\usepackage{amsfonts,amssymb,latexsym,amsthm,mathtools}
\usepackage{comment} 
\usepackage{graphicx}
\usepackage{caption} 
\usepackage{url}

\newcommand{\re}{\mathfrak{R}}
\newcommand{\im}{\mathfrak{Im}}

\newcommand{\cal}{\mathcal}
\newcommand{\scr}{\mathscr}
\newcommand{\sg}{\sigma}

\renewcommand{\d}{\displaystyle}
\renewcommand{\b}{\overline}

\newcommand{\Z}{\mathbb{Z}}
\newcommand{\R}{\mathbb{R}}
\newcommand{\Q}{\mathbb{Q}}
\newcommand{\C}{\mathbb{C}}

\newcommand{\HH}{\mathbb{H}}

\newtheorem{theo}{Theorem}[section]

\newtheorem{lemm}[theo]{Lemma}

\newtheorem{cor}[theo]{Corollary}
\newtheorem{lem}[theo]{Lemma}

\newtheorem{rem}[theo]{Remark}

\newcommand*{\Cone}{\mathscr{C}}
\newcommand*{\CDSE}[2]{\check{#1}_{#2}}

\begin{document}

\title[Period functions and Moments of a weighted mean square of $L-$functions]{
On twisted period functions and Moments of a weighted mean square of Dirichlet $L-$functions on the critical line}
\author{S\'ebastien Darses -- Berend Ringeling -- Emmanuel Royer}
\date{}

\begin{abstract}
    We extend to Dirichlet $L-$functions associated with arbitrary primitive characters a range of objects and properties -- including Eisenstein series and period functions -- that were originally introduced and studied by Lewis and Zagier (2001), and later by Bettin and Conrey (2013) in the case of the Riemann zeta function, and more recently by Lewis and Zagier (2019) for odd real characters. These tools yield closed-form expressions for the moments of a measure defined via a weighted mean square of the $L-$function. These moments not only provide a complete characterization of the modulus of the $L-$function on the critical line, but also imply an infinite number of non-trivial positivity conditions valid for all primitive characters, real or not. The methods also involve a general form of an asymptotic formula based on the shifted Euler–Maclaurin summation formula, which may be of independent interest.
\end{abstract}

\address{CNRS -- Universit\'e de Montr\'eal CRM–CNRS \newline 
Aix-Marseille Universit\'e, CNRS, I2M, Marseille, France} 
\email{seb.darses@gmail.com}

\address{ CRM, Universit\'e de Montr\'eal, P.O. Box 6128, Centre-ville Station
Montr\'eal (Qu\'ebec) H3C 3J7, Canada} 
\email{b.j.ringeling@gmail.com}
\address{CNRS -- Universit\'e de Montr\'eal CRM–CNRS \newline 
Universit\'e Clermont Auvergne, CNRS, Laboratoire de math\'matiques Blaise Pascal, F-63000 Clermont-Ferrand, France.} 
\email{emmanuel.royer@math.cnrs.fr}


\maketitle

\section{Introduction}

\subsection{Main results}

We write $e(z)=e^{2\pi iz}$  ($i^2=-1$), for a complex number $z=\mathfrak{R}(z)+i \mathfrak{Im}(z)$. The letter $q$ denotes an integer $q\geq3$, and we use the convenient notation $\xi=\xi_q=e^{2\pi i/q}$.

For a character $\chi$ of conductor $q$, the corresponding Gauss sum and Dirichlet $L-$function are defined respectively as: $\d\tau(\chi) = \sum _{a=1}^{q-1} \chi (a) \xi^a$, and 
$\d L(s,\chi) = \sum_{n=1}^\infty \frac{\chi(n)}{n^s}$, for $\mathfrak{R}(s)>1$ and by analytic continuation on $\C$. Recall that $\tau(\chi)\neq 0$ for a primitive character $\chi$. 

The letters $N,a,b,c,j,k,\ell, m,n$ denote non-negative integers and $B_j(\cdot)$ denotes the $j$th Bernoulli polynomial. The Euler totient function is denoted by $\varphi$, and the number of divisors of $n$ by $d(n)$.

Following \cite{Kou82}, we define by induction the non-central Stirling numbers of the second kind as: $S_\alpha(n+1,k)=S_\alpha(n,k-1)+(k-\alpha)S_\alpha(n,k)$, $S_\alpha(n,0)=(-\alpha)^n$ if $n\geq0$, $S_\alpha(0,k)=0$ if $k\geq1$. The non central Stirling numbers arising in our study are $S_{-1/2}(n,k)$.\\

Our main results are Theorems \ref{main-th} and \ref{deriv-psi}.

\begin{theo} \label{main-th}
For all primitive characters $\chi\mod q$, with \(q>1\), and all $N\geq0$, 
\begin{multline*}
\int_{-\infty}^\infty \left|L\left(\frac{1}{2}+it, \chi\right)\right|^{2} \frac{t^N\ dt}{\cosh(\pi t)} \   = \  -\chi(-1) (-1)^{\lfloor \frac{N}{2}\rfloor} B_N\left(\frac{1}{2}\right) \frac{\varphi(q)}{q} \\
    + \sum_{k=1}^{N+1} S_{-\frac{1}{2}}(N,k-1) \frac{(2 \pi i q)^k}{q^2 k^2} \sum_{1\leq a,b\leq q}  B_{k}\left(\frac{a}{q}\right) B_{k}\left(\frac{b}{q}\right) \scr T_{\chi,N}(ab)\ \xi^{ab}_{q}\ ,
\end{multline*}
where\  $\ \scr T_{\chi,N}(c) = -i^N \left[\tau(\b \chi)\ \chi(c) +(-1)^N \tau(\chi)\ \b \chi(c) \right]$, $c\in \Z$.
\end{theo}

A few remarks:
\begin{enumerate}
    \item 
    The formula involves only rational numbers, power of $\pi$, and $\cos,\sin$ evaluated in $2\pi\Q$. 
    Note that $(-1)^{\lfloor \frac{N}{2}\rfloor} B_N\left(\tfrac{1}{2}\right)=0$ if $N$ is odd, and $(-1)^{\lfloor \frac{N}{2}\rfloor} B_N\left(\tfrac{1}{2}\right) > 0$ if $N\geq0$ is even. The Stirling numbers involved are positive, while the entire subsequent term is real and may change sign. If $\chi$ is real, the odd moments vanish. If $\chi$ is non-real, the odd moments may be positive or negative. See Section 6 for further details.
    
    \item The weight $1/\cosh(\pi t)=|\Gamma(1/2+it)|^2/\pi$ (Reflexion formula) is not merely cosmetic; it encodes a fundamental interaction between the Gamma function and $L(\cdot,\chi)$, which leads to this closed-form expression. 
    \item A nice curiosity: $t\mapsto 2\left|L\left(\frac{1}{2}+it, \chi_{2^2}\right)\right|^{2} /\cosh(\pi t)$ is a probability density on $\R$, where $\chi_4$ is the odd character modulo $4$, applying Theorem \ref{main-th} with $\chi=\chi_4$ and $N=0$.
\end{enumerate}

Let $\im(z)>0$ and $\chi$ be a primitive character $\mod q$. Define:
\begin{eqnarray}
E_\chi(z) & = & \frac{1}{\tau(\chi)} \sum_{n= 1}^\infty \chi(n)d(n)e\left(nz/q\right) \label{Echi}\\
\psi_\chi(z) & = & E_\chi(z) - \frac{\chi(-1)}{z}\ E_{\b\chi}\left(-\frac{1}{z}\right). \label{psichi}
\end{eqnarray}
Once the tools developed in Section \ref{sec-gene} are in place -- establishing that $\psi_\chi$ has an analytic continuation to $\C\setminus (-\infty,0]$ -- the previous result is a consequence of the following identity.

\begin{theo} \label{deriv-psi}
For all primitive characters $\chi\mod q$, and all $n\geq0$,
    \begin{eqnarray}
\psi_{n,\chi}:=  \frac{\psi_\chi^{(n)}(1)}{n!} 
        = \frac{\varphi(q)}{2 \pi i q} \frac{(-1)^{n}}{n+1} + \scr B_{\chi}(n) + (-1)^n \sum_{j=0}^n {n \choose j}  \scr B_{\b\chi}(j),
    \end{eqnarray}
where 
\begin{eqnarray*}
    \scr B_\chi(j)
         & = & \frac{(2 \pi i q)^j}{\tau(\chi) (j+1)(j+1)!} 
            \sum_{1\leq a,b\leq q} B_{j+1}\left(\frac{a}{q}\right) B_{j+1}\left(\frac{b}{q}\right) \chi(ab)\ \xi^{ab}.
\end{eqnarray*}
\end{theo}

This generalizes the formula obtained by Bettin and Conrey \cite[Lemma 1, Part two]{BC13b} for the coefficient $\psi_n$ of their period function $\psi$, see Eq. (\ref{psi-k}) here. However, our proof relies on a different approach than that of \cite{BC13b}, as explained below.

The curious reader will be pleased to recover $\psi_n$ from $\psi_{n,\chi}$ by specializing to $q=1$, $a=b=1$, and $\chi(1)=1$ in Theorem \ref{deriv-psi}.

\bigskip

\subsection{Why may these identities be interesting/useful?}

\begin{enumerate}
    \item The moments in Theorem \ref{main-th} arise in {\em generalizations of the Nyman–Beurling–B\'aez-Duarte criteria}. These criteria are approximation problems in $L^2(0,\infty)$ that provide reformulations of the Riemann Hypothesis (RH). They can be generalized in two directions: on the one hand, by altering the structure of the approximation problem related to the Riemann zeta function (see \cite{DH21b,ADH22}); and on the other hand, by considering generalizations of the $\zeta$ function itself (see \cite{DFMR13, LZ19}). In particular, Lewis and Zagier studied in \cite{LZ19} a concrete auto-correlation function (or inner product) involving the primitive character of conductor $4$, and derived an equivalent formulation of the Generalized Riemann Hypothesis for odd characters. The quantity we investigate here is related to a different structure of the inner product; see \cite{DH24} for further details.
\smallskip
    \item Theorem \ref{main-th} constitutes a {\em determinate Hamburger moment problem} — that is, the right-hand side provides a complete characterization, in a discrete setting, of $|L(1/2+it,\chi)|$, similar to the case of $\zeta$ (cf. \cite{DH24}). A basic upper bound on $L(1/2+it,\chi)$ (see e.g. \cite[Chap. 12, (14)]{Dav80}) is enough to ensure that the moments satisfy a standard determinacy criterion (cf. \cite[Prop. 1.5, p. 88]{Sim98}). The central question then becomes how to extract information about the behavior of the $L$-function, or about specific values, particularly for small $t$, as suggested by O. Ramaré.
    \item The computation leading to Theorem \ref{deriv-psi} is a {\em nontrivial application of an asymptotic formula} for sums of the form $\sum_n f((n+\alpha)t)$ — a method described by Zagier as “extremely useful, and not sufficiently well known” (see \cite[p.11]{Zag06}). This approach relies on a {\em shifted Euler–Maclaurin summation formula}. For our purposes, we require a slight extension of this formula to some complex domain, which is presented in detail in Section \ref{sec-mac}, and may be of independent interest.
    \smallskip
    \item The formulas presented here are {\em closed-form identities} — that is, finite sums of tabulated quantities — which allow for the evaluation of the integral with near-infinite precision. For numerous insightful examples, see \cite{BoC13}. To the best of our knowledge, the first closed-form expression involving an integral of an $L$-function was given by Lewis and Zagier for $\chi_4$, by combining Proposition 1, Proposition 2 (p.6), Equation (21), and the Mellin isometry in \cite{LZ19}. 
    \smallskip
    \item 
    Moreover, the variety of tools involved, and the numerous connections with modular forms, Fourier analysis, combinatorics, and special functions, may be of independent interest; see e.g. \cite{BLZ15, LZ19, AIK14, Com12}. The toolbox developed in Section \ref{sec-gene} can also be used to extend to $L$-functions the sixth moment formula obtained for $\zeta$ in \cite{DN24}, as suggested by  Y. Lamzouri.\smallskip 
    \item Last but not least, considering $2N$ in Theorem \ref{main-th}, the positive right hand side yields a {\em sequence of non trivial positive sums} involving all primitive characters, wether real or not.\\
    For any {\em even} $N$, the number $\chi(-1)(-1)^{\lfloor \frac{N}{2}\rfloor} B_N\left(\tfrac{1}{2}\right)\frac{\varphi(q)}{q}$, depending on its sign, then provides either a lower or an upper bound, in terms of $q$ and $N$, for the whole sum in the formula.
\end{enumerate}

\subsection{Previous works and main ideas} 

\subsubsection{A formula in the case of $\zeta$}
In \cite{DH24} it is proven that for all 
$N\geq 0$, writing $\sum_\varnothing :=0$,
\begin{eqnarray*}
\frac{(-4)^N}{2}\int_{-\infty}^\infty \left|\zeta\left(\frac{1}{2}+it\right)\right|^2 \frac{t^{2N} \ dt}{\cosh(\pi t)}  = \log(2\pi)-\gamma -4N + \left(\frac{4^N}{2}-1\right) B_{2N} + \sum_{j=2}^{2N} T_{2N,j}\frac{\zeta(j)B_{j}}{j}, \label{moment-poly}
\end{eqnarray*} 
where
$
T_{N,j} = (j-1)!\sum_{2\leq n \leq N} \binom{N}{n} 2^{n} \left[(-1)^n S(n+1,j) + (-1)^j S(n,j-1)\right]$. This formula can be reformulated by means of the non central Stirling number $S_{-1/2}(N,j)$.
\smallskip

A generalization of this formula
has been anticipated independently by C. Delaunay and O. Ramaré \cite{Ram22}, for $L-$functions of quadratic character, at least.

\subsubsection{Lewis-Zagier and Folsom twisted period functions} Our definitions (\ref{Echi}-\ref{psichi}) are consistent with the functions introduced by Lewis and Zagier in \cite[p.21]{LZ19} in the case of the primitive character $\chi_4$ of conductor $4$ (the authors also notice that their study can be generalized to primitive odd Dirichlet character $\chi_D$). Indeed, they define 
\begin{eqnarray*}
    f(z) = \sum_{n\ge1}\chi_4(n)d(n)e\left(nz/4\right),\qquad 
    \psi_f(z) = f(z) + \frac{1}{z} f(-1/z).
\end{eqnarray*}
Note that $\chi_4(-1)=-1$, $\b\chi_4=\chi_4$ and $\tau(\b \chi_4)=\tau(\chi_4)$. Hence we see that, up to the common factor $\tau(\chi)^{-1}$, our $E_{\chi_4}$ (resp. $\psi_{\chi_4}$) is their $f$ (resp. $\psi_f$).
Let us notice that the factors $\tau(\chi)^{-1}$, $\chi(-1)$ and the conjugation $\b\chi$ in (\ref{Echi}-\ref{psichi}), will appear as reasonable choices to obtain an exact relation between $\psi_\chi$ and $A_\chi$, see the proof in Section \ref{bettin-conrey}.

A. Folsom \cite{Fol20} also proposed another large class of twisted Eisenstien series, based on general twisted divisor functions this time, and manage to show the analytic continuation of the associated period functions to $\C'$.

\subsubsection{Bettin-Conrey works in the case of $\zeta$ and the connection with an identity of Ramanujan}
The function $A$ introduced in \cite{DH21a}, is studied for complex numbers in \cite{DN24}:
\begin{eqnarray*}
A(z) & = & \int_0^\infty \left(\frac{1}{xz}-\frac{1}{e^{xz}-1}\right)\left(\frac{1}{x}-\frac{1}{e^{x}-1}\right)dx, \qquad \re(z)>0.
\end{eqnarray*}
It is an auto-correlation function and shows up in a particular generalization of the Nyman-Beurling criterion \cite{DH21b} (see \cite{BDBLS05} for a study of the auto-correlation function in the historical criterion).

The function $A$ has an analytic continuation to $\C\setminus (-\infty,0]$, and
allows to connect continuous and discrete Fourier analysis related to $\zeta$ through the following reformulations of identities by Ramanujan (\cite{Ram15}, 1915) and by Bettin and Conrey (\cite{BC13b}, 2013), respectively:
\begin{eqnarray}
\scr R(z) := e^{z/2} A(e^z) & = &  \frac{1}{2} 
\int_{-\infty}^{\infty} e^{i z t}\left|\zeta\left(\frac{1}{2}+it\right)\right|^{2} \frac{dt}{\cosh(\pi t)}, \quad -\pi<\im(z)<\pi, \\
A(z) & = & \frac{i\pi}{4} \psi(z) + r(z), \quad z\in \C\setminus (-\infty,0],
\end{eqnarray}
where
$\d r(z)  = \frac{\log(2\pi)-\gamma}{2}\left(\frac{1}{z}+1\right) + \frac{1}{2}\left(\frac{1}{z}-1\right) \log(z)$, and the weight one Eisenstein series and period function defined in \cite{BC13a,BC13b} read: for $\im(z)>0$, 
\begin{eqnarray*}
E_1(z) = 1-4\sum_{n=1}^\infty d(n)e(n z), \qquad
\psi(z) = E_1(z)-1/z\ E_1(-1/z).
\end{eqnarray*}
Based on ideas in \cite{LZ01} they prove that $\psi$ has an analytic continuation to $\C\setminus (-\infty,0]$, and moreover, see 
\cite[Lemma 1, Part two]{BC13b}, for $|z|<1$, 
\begin{eqnarray} \label{psi-k}
\psi(1+z) =\frac{2i}{\pi}\sum_{k=0}^\infty \psi_k z^k, \quad \psi_k = \frac{(-1)^k}{k+1} + 2 \sum_{j=1}^{k-1}(-1)^{k-j}{k \choose j} \frac{\zeta(j+1)B_{j+1}}{j+1}.
\end{eqnarray}

To obtain the moments in our study, a
natural method is then to differentiate $z\mapsto e^{z/2} A(e^z)$ at $0$ and so $A$ at $1$, and use Eq. (\ref{psi-k}). We provide here all the necessary tools to do so in the case of the $L-$function. 
To prove the counter-part of (\ref{psi-k}) for $\psi_\chi$, we use an original method based on a double application of a quite general version of the shifted Euler-Maclaurin summation formula discussed by Zagier \cite{Zag06}, which may be one of the interests of this paper.

\subsection{Additional notations, properties}

We define $\R_{>0}=(0,\infty)$, $\R_{\geq 0}=[0,\infty)$, $\R_{<0}=(-\infty,0)$ and so on. We will work with different complex domains: the Poincar\'e upper half-plane $\HH=\{x+iy, y>0\}$, the right half-plane $\re_{>0}=\{x+iy, x>0\}$, and $\C'=\C\setminus (-\infty,0]$. 

The fractional part of a real number \(x\) is \(\{x\}=x-\lfloor x\rfloor\) where \(\lfloor x\rfloor\) is the largest integer less than or equal to \(x\). For example, for any \(\alpha\in(0,1)\), on has
\[
\{x-\alpha\}=\begin{cases*}
    x-\alpha+1 & if \ \(0<x\leq\alpha\),\\
    x-\alpha & if \ \(\alpha<x\leq 1\).
\end{cases*}
\]

The Gamma function is defined as: $\d
\Gamma(s) = \int_0^\infty e^{-x}x^{s-1}dx, \quad \mathfrak{R}(s)>0$.

The Fourier transform $\cal F[f]$ and Mellin transform $\cal M[g]$ are defined as:
\begin{eqnarray*}
    \cal F[f](w) & = & \int_{-\infty}^\infty e^{-ixw}f(x)dx,\quad w\in\R, \\
    \cal M[g](s) & = & \int_0^\infty g(x)x^{s-1}dx, \quad 0<\re(s)<1,
\end{eqnarray*}
for $f\in L^1(\R)$ and $g:\R^+\to\R$ such that $\d \int_0^\infty |g(x)|x^{\sg-1}dx<\infty$ for all $\sg\in(0,1)$.

We will also encounter the general Gauss sum
\begin{eqnarray*}
    \tau(n,\chi ) = \sum _{a=1}^{q-1} \chi (a) \xi^{an} = \b \chi(n)\tau(\chi).
\end{eqnarray*}
Recall that $\b{\tau(\chi)}=\chi(-1)\tau(\b \chi)$. 

The Bernoulli polynomials $B_n$ are defined as $\d B_n(x)=\sum_{k=0}^n {n \choose k} B_{n-k}x^k$ where $B_k$ are the $k-$th Bernoulli number: $B_0=1,\ B_1=-1/2,\ B_2=1/6,\ B_3=0$, etc. 

See e.g. \cite[p.189]{Ap08} for a review of many important relations.

\subsection{Outline}
Section \ref{sec-gene} is devoted to the study of the various generalizations $A_\chi,\psi_\chi$ of the functions $A$ and $\psi$. Sections \ref{sec-proof-th2} and \ref{sec-proof-th1} contain the proofs of Theorems \ref{deriv-psi} and \ref{main-th}, respectively. Section \ref{sec-mac} presents a general form of the shifted Euler–Maclaurin summation formula, which is used in the proof of Theorem \ref{deriv-psi}. Finally, Section \ref{sec-num} collects several numerical experiments.

\bigskip

\section{The various generalizations, their relations and analytic properties}
\label{sec-gene}

\subsection{The generalization $A_\chi$}

The following relation, for $0<\mathfrak{R}(s)<1$,
\begin{eqnarray*}\label{mellin-gamma-zeta}
\Gamma(s)\zeta(s) = \int_0^\infty f(x) x^{s-1}dx , \quad f(x) = \sum_{n=1}^\infty e^{-nx} - \frac{1}{x} = \frac{1}{e^{x}-1}-\frac{1}{x},\ x>0,
\end{eqnarray*}
can be generalized to $L-$functions through the classical representation, see e.g. \cite{IR90} p.263,
\begin{eqnarray} \label{mellin-g-L}
\Gamma(s) L(s,\chi) & = & \int_0^\infty f_\chi(x) x^{s-1}dx, \quad \re(s)> 0,
\end{eqnarray} 
where \begin{eqnarray} \label{fchi}
    f_\chi(x) & = & \sum_{n=1}^\infty \chi(n) e^{-nx}, \quad x>0.
\end{eqnarray}

A possible generalization of the auto-correlation function $A$ is then:
\begin{eqnarray} \label{A-chi}
A_\chi(v) & = & \int_0^\infty f_\chi(xv) f_{\b \chi}(x) dx,\quad v>0.
\end{eqnarray}
The reason for conjugating one factor will appear in the proof of Lemmas \ref{Achi-Bchi} and \ref{fund-lemma}.

We provide a few basic properties of $f_\chi$ and $A_\chi$, which will be useful for the proofs. 

\begin{lemm} \label{f-chi}
The functions $f_\chi$ and $A_\chi$ can be extended to continuous functions at 0. Moreover, $f_\chi(x)=O(e^{-x})$  as $x\to+\infty$, and $A_\chi(v)=O(1/v)$  as $v\to+\infty$.
In particular, $\cal M[A_\chi](s)$ is well defined for $0<\re(s)<1$. 
\end{lemm}

\begin{proof}
The continuity of $f_\chi$ at $0$ is obtained, for instance, by the analytic continuation given by 
\begin{eqnarray} \label{dev-fchi}
    f_\chi(z) & = & -\sum_{n=0}^\infty (-1)^n\frac{B_{n+1,\chi}}{(n+1)!}z^n, \quad |z|<2\pi/q,
\end{eqnarray}
where $(B_{n,\chi})_{n\geq 0}$ is defined by: $\d \sum_{a=1}^{q-1}\chi(a)\frac{x e^{ax}}{e^{qx}-1}=\sum_{n= 0}^\infty B_{n,\chi}\frac{x^n}{n!}$, see \cite[p.264]{IR90}. 
A crude bound from the definition (\ref{fchi}) yields the bound at infinity for $f_\chi$.

The dominated convergence theorem provides the continuity of $A_\chi$ at $0$ for any character, and the functional equation
\begin{eqnarray*} 
    A_\chi(v) = A_{\b\chi}(1/v) / v = \overline{A_{\b\chi}(1/v)} / v, \quad v>0,
\end{eqnarray*}
gives the bound at infinity for $A_\chi$, which allows to conclude.
\end{proof}

\medskip

\subsection{A relation between $A_\chi$ and $L(\cdot,\chi)$: Analytic properties}

We now propose a natural generalization of the function $\scr R$, introduced by Ramanujan in \cite{Ram15}, and of the function $Q$, introduced by Bettin and Conrey in \cite{BC13b}:
\begin{eqnarray*}
    \scr R_\chi(x) & = & e^{x/2}A_\chi(e^x)\\
    Q_\chi(s) & = & \Gamma(s) L(s,\chi) \Gamma(1-s) L(1-s,\b \chi) \ =\  \frac{\pi}{\sin(\pi s)}L(s,\chi)L(1-s,\b\chi). 
\end{eqnarray*}

The following lemma generalizes the formulation given in \cite{DN24}, adapting {\em mutatis mutandis} the proof in the case of $A_\chi,\scr R_\chi, Q_\chi$. We only give some essential details.

\begin{lemm} \label{Achi-Bchi}
We have
\begin{eqnarray} \label{mellinA}
\cal M [A_\chi](s) & = & Q_\chi(s), \quad 0<\mathfrak{R}(s)<1.
\end{eqnarray}
Moreover, $\scr R_\chi$, respectively $A_\chi$, has an analytic continuation on $\cal W=\{x+iy, -\pi<y<\pi\}$, resp. on  $\C'=\C\setminus (-\infty,0]$, given by:
\begin{eqnarray}
\scr R_\chi(w) & = & \frac{1}{2}\int_{-\infty}^\infty e^{itw}\left|L\left(\frac{1}{2}+it, \b \chi\right)\right|^{2} \frac{dt}{\cosh(\pi t)}, \quad w\in \cal W, \\
A_{\chi}(z) & = & z^{-1/2}\scr R_\chi(\log z) \ = \ \int_{-\infty}^\infty z^{-\frac{1}{2}+it} \left|L\left(\frac{1}{2}+it, \b \chi\right)\right|^{2} \frac{dt}{\cosh(\pi t)}, \quad z\in \C'.
\end{eqnarray}

\end{lemm}

\begin{proof}
The relevant integrability conditions can be checked as in \cite{DN24}.
For $0<\mathfrak{R}(s)<1$, writing $-s=(1-s)-1$ and using (\ref{mellin-g-L}), we can write
\begin{eqnarray*}
\cal M [A_\chi](s) & = & \int_0^\infty \int_0^\infty f_\chi(xv)v^{s-1}dv\  f_{\b \chi}(x)dx \\
    & = & \int_0^\infty f_\chi(v)v^{s-1}dv \int_0^\infty f_{\b \chi}(x) x^{-s}dx \ = \ Q_\chi(s).
\end{eqnarray*}
Moreover, using the change of variable $y=e^x$, we obtain 
\begin{eqnarray*}
  \cal F [\scr R_\chi](t) = \int_{-\infty}^{\infty} e^{-i t x} 
e^{x/2} A_\chi(e^x) dx =  \int_0^{\infty} y^{1/2-it} A_\chi(y) \frac{dy}{y} = \cal M[A_\chi]\left(\frac{1}{2} - it\right).
\end{eqnarray*}
Therefore, (\ref{mellinA}) and Euler's reflection formula yield:
\begin{eqnarray} \label{fourierB}
\cal F [\scr R_\chi](t) & = & \left|L\left(\frac{1}{2}+it,\b \chi\right)\right|^{2} \frac{\pi}{\cosh(\pi t)}, \quad t\in\R.
\end{eqnarray}
We apply the inverse Fourier transform to (\ref{fourierB}) to obtain $\scr R_\chi$ and its analytic continuation to $\cal W$. Finally, we write $A_{\chi}(y)=y^{-1/2}\scr R_\chi(\log y)$, for $y>0$, and note that $\log(\C')=\cal W$.
\end{proof}

\bigskip

\subsection{A relation between $A_\chi$ and $\psi_\chi$} \label{bettin-conrey}

Recall that for $\im(z)>0$, 
\begin{eqnarray*}
E_\chi(z) & = & \frac{1}{\tau(\chi)} \sum_{n=1}^\infty\chi(n)d(n)e\left(nz/q\right) \\
\psi_\chi(z) & = & E_\chi(z) - \frac{\chi(-1)}{z}\ E_{\b\chi}\left(-\frac{1}{z}\right),
\end{eqnarray*}
where $\chi$ is a primitive character (see \cite[p.318]{Ten22} for the general condition for which $\tau(\chi)\neq 0$).

We basically use the same tricks as in the proof of \cite[Lemma 1]{BC13b} (see also \cite{DN24}), carefully keeping track of $\chi$ and its modulus $q$, to prove the following generalization.

\begin{lemm} \label{fund-lemma}
Let $\chi$ be a primitive Dirichlet character modulo $q>1$.\\
The analytic continuations of $A_\chi$ and $\psi_\chi$ satisfy : 
\begin{eqnarray*}
  A_\chi(z) & = & -\chi(-1) i\pi\ \psi_\chi(z), \quad z \in \mathbb{C}'.
\end{eqnarray*}
\end{lemm}

\begin{proof}
Let $z\in\HH$, then $-iz=e^{-i\pi/2}z\in\re_{>0}$ and $|\arg(-2\pi i z)|<\pi/2$. 
Using the expansion 
\begin{eqnarray*}
L^2(s,\chi) & = & \sum_{n= 1}^\infty\frac{\chi(n)d(n)}{n^s},
\end{eqnarray*}
and the functional equation 
\begin{eqnarray*}
L(s,\chi) & = & i^{-\delta} \tau(\chi) 2^s \pi^{s-1}q^{-s} \sin(\pi (s+\delta)/2)\Gamma(1-s)L(1-s, \b\chi),
\end{eqnarray*}
where $\chi(-1)=(-1)^{\delta}$, we obtain
\begin{eqnarray*}
    E_\chi(z) & = & \frac{\tau(\chi)^{-1}}{ 2\pi i} \int_{(2)} \Gamma(s) L(s,\chi)^2 (-2\pi i z/q)^{-s} ds \\
        & = & \frac{i^{-\delta}}{ 2\pi i} \int_{(2)} \Gamma(s) L(s,\chi) 2^s \pi^{s-1}q^{-s} \sin(\pi (s+\delta)/2)\Gamma(1-s)L(1-s,\b\chi)(-2\pi i z/q)^{-s} ds\\ 
        & = & \frac{i^{-\delta}}{ 2\pi^2 i} \int_{(2)}  \sin(\pi (s+\delta)/2) e^{\pi i s/2} Q_\chi(s) z^{-s} ds. 
\end{eqnarray*}
We recall that $Q_\chi(s) = \Gamma(s) L(s,\chi) \Gamma(1-s) L(1-s,\b \chi)$ and we set
\begin{eqnarray*}
    \scr Q_\chi(s,z) & = & \sin(\pi (s+\delta)/2) e^{\pi i s/2} Q_\chi(s) z^{-s}.
\end{eqnarray*}
Since $\scr Q_\chi(\cdot,z)$ has exponential decay at $\pm i \infty$, and $\scr Q_\chi(\cdot,z)$ has no poles in $\re(s)>0$ (which can be seen directly in the previous first expression of $E_\chi(z)$), we can move the line of integration to $(1/2)$ and get
\begin{eqnarray*}
    E_\chi(z)   & = &  \frac{i^{-\delta}}{ 2\pi^2 i} \int_{(1/2)} \scr Q_\chi(s,z) ds.
\end{eqnarray*}
Therefore, for $z \in \mathbb{H}$, i.e. $0<\arg z<\pi$,
\begin{eqnarray*}
    - \frac{1}{z}\ E_{\b\chi}\left(-\frac{1}{z}\right) & = &  -\frac{1}{z}\cdot \frac{i^{-\delta}}{ 2\pi^2 i} \int_{(1/2)} \scr Q_{\b \chi}(s,-1/z) ds.
\end{eqnarray*}
We have, using the change of variable $s\to 1-s$,
\begin{eqnarray*}
    -\frac{1}{z}\int_{(1/2)} \scr Q_{\b \chi}(s,-1/z) ds & = & -\frac{1}{z}\int_{(1/2)} Q_{\b\chi}(s) \sin(\pi (s+\delta)/2) e^{\pi i s/2} (-1/z)^{-s} ds \\
        & = & \int_{(1/2)} Q_{\b\chi}(1-s) \sin(\pi (1-s+\delta)/2) e^{\pi i (1-s)/2} (-1/z)^{s} ds
\end{eqnarray*}
But $Q_{\b\chi}(1-s)=Q_\chi(s)$ and $$\sin(\pi (1-s+\delta)/2)=\cos(\pi (s-\delta)/2).$$ 
Moreover $(-1/z)^s=e^{\pi i s}z^{-s}$ since $\arg(-1/z)=\pi -\arg z$, and then $e^{\pi i (1-s)/2} e^{\pi i s} = i e^{\pi i s/2}$. Therefore
\begin{eqnarray*}
    - \frac{1}{z}\ E_{\b\chi}(-\frac{1}{z}) & = &  \frac{i^{-\delta}}{ 2\pi^2 i} \int_{(1/2)}  \cos(\pi (s-\delta)/2)\ i  e^{\pi i s/2}  Q_\chi(s) z^{-s} ds.
\end{eqnarray*}
We now compute
\begin{eqnarray*}
\sin(\frac{\pi (s+\delta)}{2}) e^{i\frac{\pi s}{2}}+ (-1)^\delta i \cos(\frac{\pi (s-\delta)}{2}) e^{i\frac{\pi s}{2}} & = & 
\frac{e^{i\frac{\pi (s+\delta)}{2}}-e^{-i\frac{\pi (s+\delta)}{2}}}{2i} e^{i\frac{\pi s}{2}} +  (-1)^\delta i \frac{e^{i\frac{\pi (s-\delta)}{2}} + e^{-i\frac{\pi (s-\delta)}{2}}}{2} e^{i\frac{\pi s}{2}} \\
    & = & -\frac{i}{2} (e^{i\pi s} e^{i\frac{\pi \delta}{2}} - e^{-i\frac{\pi \delta}{2}}) + (-1)^\delta \frac{i}{2} (e^{i\pi s} e^{-i\frac{\pi \delta}{2}} + e^{i\frac{\pi \delta}{2}})\\
    & = & \frac{i}{2} e^{i\pi s} \left((-1)^\delta e^{-i\frac{\pi \delta}{2}} -e^{i\frac{\pi \delta}{2}}\right)  + \frac{i}{2} (e^{-i\frac{\pi \delta}{2}} + (-1)^\delta e^{+i\frac{\pi \delta}{2}})\\
    & = & i (-i)^\delta,
\end{eqnarray*}
where we noted that, either if $\delta=0$ or $1$, then $(-1)^\delta e^{-i\frac{\pi \delta}{2}} -e^{i\frac{\pi \delta}{2}}=0$.
We then obtain
\begin{eqnarray*}
\psi_\chi(z)  =  E_\chi(z) - \frac{(-1)^\delta}{z}\ E_{\b \chi}\left(-\frac{1}{z}\right) = 
    \frac{i^{-\delta}}{ 2\pi^2 i} \int_{(1/2)} i (-i)^\delta Q_\chi(s) z^{-s} ds = \frac{i\chi(-1)}{\pi} A_\chi(z),
\end{eqnarray*}
and we conclude due to Lemma \ref{Achi-Bchi}.
\end{proof}

Another proof can be performed, in the same vein of \cite[p.22]{LZ19}, computing directly the Mellin transform of $t\mapsto E_\chi(it)$. To do so, we would need to study the behavior at $0$. Writing, for $t>0$,
\begin{eqnarray*}
  \tau(\chi) E_{\chi}(it) & = & \sum_{n = 1}^\infty \chi(n) \sum_{b = 1}^\infty \chi(b) e^{-b\cdot 2 \pi t\ n/q} \\  
  &=& \sum_{n = 1}^\infty \chi(n) f_{\chi}(2 \pi n t/q) \\ 
  &=& \sum_{a = 1}^{q-1}\chi(a) \sum_{n = 0}^\infty f_{\chi}(2 \pi (n + a/q)t),
\end{eqnarray*}
we can apply the shifted
Euler-Maclaurin summation formula of Theorem \ref{main-shift} to the function $x\mapsto f_\chi(2\pi x)$ with shift $a/q$, using the expansion of $f_\chi$ given by (\ref{dev-fchi}).

We then remark that $t\mapsto E_\chi(it)$ has a continuous extension at $0$ since
\[
\tau(\chi) E_{\chi}(i t) = \frac{c_{-1}}{t} + c_{0} +o(t), \quad t \to 0^+,
\]
where
\begin{eqnarray*}
c_{-1} = \sum_{a = 1}^{q-1} \chi(a) \int_{0}^\infty f_{\chi}(t) \frac{dt}{2\pi} = 0.
\end{eqnarray*}

\bigskip

\section{Proof of Theorem \ref{deriv-psi}} \label{sec-proof-th2}

The expansion of
\begin{eqnarray*}
    \psi_\chi(1+it) & =  & E_{\chi}(1+it) - \frac{\chi(-1)}{1+it} E_{\b \chi}\left(-\frac{1}{1+it}\right)
\end{eqnarray*}
as $t\to 0^+$ will provide the derivatives $\psi_\chi^{(n)}(1)$ at any order.

We need to compute $\d E_{\chi}(1+it)$ and $$E_{\b \chi}\left(-\frac{1}{1+it}\right) = E_{\b \chi}\left(-1+i\frac{t}{1+it}\right).$$
To do so, we study $ E_{\chi}(\pm1 + iw)$ with $w\in\re_{>0}$. First let us prove the following decomposition.

\begin{lemm}
We have for all $w\in\re_{>0}$:
\[
 E_{\chi}(\pm1+ iw) = \frac{1}{\tau(\chi)} \sum_{j = 1}^{q-1} \chi(j) \sum_{k \geq 0} \cal E_{\chi, \pm j}((j/q + k)\ iw),
\]
where $\d \cal E_{\chi, \pm j}(z) = \sum_{b \geq 1} \chi(b) 
e(\pm bj/q) e\left( b z\right)$ for $z\in\HH$.
\end{lemm}
\begin{proof}
First, using the definition of the divisor function and summing according to congruence classes $\mod q$, we can write:
\begin{eqnarray*}
\tau(\chi) E_{\chi}(z) 
& = &  \sum_{m \geq 1} \chi(m) \sum_{b \geq 1} \chi(b) e\left( b m z/q\right) \\
& = & \sum_{j = 1}^{q -1 } \sum_{\substack{m \geq 1\\ m \equiv j\mod q}} \chi(m) \sum_{b \geq 1} \chi(b) e\left( b m z/q\right)\\
& = & \sum_{j = 1}^{q -1 } \chi(j) \sum_{k \geq 0} \sum_{b \geq 1} \chi(b) 
e\left( b(j + k q) z/q\right).
\end{eqnarray*}
Therefore
$\d \ \tau(\chi) E_{\chi}(z \pm1 ) = \sum_{j = 1}^{q -1 } \chi(j) \sum_{k \geq 0} \sum_{b \geq 1} \chi(b) 
e(\pm b j/q)\  e\left( b(j/q + k ) z\right)$,
as claimed.
\end{proof}

We then need to get the expansion of $\d \sum_{k \geq 0} \cal E_{\chi, r}(i(j/q + k)w)$, which is the purpose of the following.

\begin{lemm} 
We have for all $\pm r\in\{1,\ldots,q-1\}$, and $w\in\re_{>0}$ near $0$,
\begin{eqnarray*}
 \sum_{k \geq 0} \cal E_{\chi, r}(i(j/q + k)w) & = & -\frac{\tau( r,\chi)}{2 \pi q} \frac{\log(w)}{w}\\
 &  & + \frac{1}{w} \left( \frac{\tau( r,\chi)}{2 \pi q} K(j/q) + I_\chi(r) \right)  \\
    &  & + \sum_{n = 0}^N \frac{B_{n+1}(j/q)}{(n+1)} c_{n,\chi}(r) w^n + O(w^N),
\end{eqnarray*}
where $\d I_{\chi}(r) = \int_{0}^\infty \left(\cal E_{\chi, r}(ix) - \frac{\tau( r,\chi)}{2 \pi q} \frac{e^{-x}}{x}\right) d x$, $\d \ K(\alpha)=- \gamma- \frac{\Gamma'(\alpha)}{\Gamma(\alpha)}$, and
\[
c_{n,\chi}(r) = \frac{(-2 \pi q)^n}{(n+1)!} \sum_{a = 1}^{q-1} \chi(a)e\left( r a /q \right) B_{n+1}(a/q).
\]
\end{lemm}

\begin{proof}
This quantity is of the form of the Euler-Maclaurin formula Theorem \ref{main-shift} with a shift $j/q$. So let us study the expansion of $\cal E_{\chi, r}(i w)$.

We write by definition, with $b=a+qc$,
\[
\cal E_{\chi,  r}(iw) = \sum_{b \geq 1} \chi(b) 
e( br/q) e\left( b iw\right) 
= \sum_{a = 1}^{q -1} \chi(a)  e\left( r a /q \right)  \sum_{c \geq 0} e\left( (a + qc)iw\right).
\]
First, we apply the shifted Euler-Maclaurin summation formula in Theorem \ref{main-shift} to the function $w\mapsto e^{-2 \pi q\ w}$ with shift $a/q$, to get:
\begin{eqnarray*}
\sum_{c \geq 0}  e^{-2 \pi (a/q + c)(qw)} & = & \frac{1}{2 \pi q\ w} - \sum_{m = 0}^N B_{m+1}(a/q) \frac{(-2 \pi q)^m}{(m+1)!} w^m +O(w^{N+1}),
\end{eqnarray*}
as $t \to 0^{+}$. Notice that $\d \tau(r,\chi)=\sum_{a = 1}^{q -1} \chi(a)  e\left( r a /q \right)$. Hence, summing over $a$, we get
\begin{eqnarray*}
\cal E_{\chi, r}(iw) & = &  
\frac{\tau( r,\chi)}{2 \pi q\ w} - \sum_{m = 0}^N c_{m,\chi}(r) w^m +O(w^{N+1}),
\end{eqnarray*}
where 
\begin{eqnarray*}
c_{m,\chi}( r) & = & \frac{(-2 \pi\ q)^m}{(m+1)!} \sum_{a = 1}^{q-1} \chi(a)e\left( r a /q \right) B_{m+1}(a/q).
\end{eqnarray*}
Second, we apply the shifted Euler-Maclaurin summation formula with shift $j/q$ and $w\mapsto \cal E_{\chi, r}(iw)$ to obtain the conclusion.
\end{proof}

We are now ready to prove the following expansion.
\begin{lemm} \label{Epm1}
We have for $w\in\re_{>0}$ near $0$:
\begin{eqnarray*}
 E_{\chi}(\pm1 +i w) 
    & = & - \chi(\pm 1) \left(\frac{\varphi(q)}{2\pi q}\frac{\log(w)}{w} + \frac{\kappa(q)}{w}\right)
     + \sum_{n = 0}^N C_{\pm1,\chi}(n)\ w^n + O(w^{N+1}),
\end{eqnarray*}
where 
\begin{eqnarray*}
C_{1,\chi}(n) & = & \frac{(-2 \pi q)^n}{\tau(\chi)(n+1)(n+1)!} \sum_{1\leq a,b \leq q-1} \chi(a b)\ \xi^{ab} B_{n+1}(a/q)B_{n+1}(b/q) \\
C_{-1,\chi}(n)  & = & (-1)^{n+1} \chi(-1) C_{1,\chi}(n),
\end{eqnarray*}
and, denoting by $\gamma$ the Euler constant and $\chi_0$ the principal character modulo $q$,
\[
\kappa(q)  = \gamma\frac{\varphi(q)}{2\pi q} + \frac{1}{2\pi q} \sum_{\substack{1\leq j\leq q\\(j,q)=1}}\frac{\Gamma'}{\Gamma}\left(\frac{j}{q}\right) - \int_0^\infty \left(f_{\chi_0}(2\pi x) - \frac{\varphi(q)}{2\pi q}\frac{e^{-x}}{x} \right)dx.
\]
\end{lemm}
\begin{proof}
Applying the two previous lemmas gives for $w\in\re_{>0}$ near $0$:
\begin{eqnarray*}
 E_{\chi}(\pm1 +i w)
    & = &  - \frac{1}{\tau(\chi)} \sum_{j = 1}^{q-1} \chi(j) \frac{\tau( \pm j,\chi)}{2 \pi q} \frac{\log(w)}{w}\\
    & &   + \frac{1}{w} \frac{1}{\tau(\chi)} \sum_{j = 1}^{q-1} \chi(j) 
    \left( \frac{\tau( \pm j,\chi)}{2 \pi q} K(j/q) + I_\chi(\pm j) \right)  \\
    & &  + \frac{1}{\tau(\chi)} \sum_{j = 1}^{q-1} \chi(j) \sum_{n = 0}^N \frac{B_{n+1}(j/q)}{(n+1)} c_{n,\chi}(\pm j) w^n + O(w^{N+1}).
\end{eqnarray*}
The equality $\d \chi(j) \frac{\tau( \pm j,\chi)}{\tau(\chi)}=\chi(\pm 1)$ if $(j,q)=0$, and $0$ if not, allows to simplify the first two sums.

From the expressions
\begin{eqnarray*}
c_{n,\chi}(r) & = & \frac{(-2 \pi q)^n}{(n+1)!} \sum_{a = 1}^{q-1} \chi(a)e\left( r a /q \right) B_{n+1}(a/q),
\end{eqnarray*}
and
\begin{eqnarray*}
    \sum_{a = 1}^{q-1} \chi(a)e\left( -r a /q \right) B_{n+1}(a/q) & = & \sum_{a = 1}^{q-1} \chi(q-a)e\left( -r (q-a) /q \right) B_{n+1}((q-a)/q) \\
        & = & (-1)^{n+1} \sum_{a = 1}^{q-1} \chi(-a)e\left( r a /q \right) B_{n+1}(a/q),
\end{eqnarray*}
we get $\chi(-1) c_{n,\chi}(-r) =(-1)^{n+1}c_{n,\chi}(r)$, and the conclusion by computing the term in $1/w$.
\end{proof}
\smallskip

We now finish the proof of Theorem \ref{deriv-psi}. From Lemma \ref{Epm1}, we can write for $t>0$ near $0$:
\begin{eqnarray*}
E_{\chi}(1+it) & = & -\frac{\varphi(q)}{2\pi q}\frac{\log(t)}{t} - \frac{\kappa(q)}{t}  + \sum_{n = 0}^N C_{1,\chi}(n) t^n + O(t^{N+1}),
\end{eqnarray*}
and
\begin{eqnarray*}
\frac{\chi(-1)}{1+it} E_{\b\chi}\left(-1+i\frac{t}{1+it}\right) & = & 
-\frac{\varphi(q)}{2\pi q}\frac{\log(t)}{t} + \frac{\varphi(q)}{2\pi q}\frac{\log(1+it)}{t} - \frac{\kappa(q)}{t} \\
    &  & + \sum_{n = 0}^N \b\chi(-1) C_{-1,\b \chi}(n) \frac{t^n}{(1+it)^{n+1}} + O(t^{N+1}).
\end{eqnarray*}
Summing, the divergent terms cancel since we know that $\psi_\chi$ is analytic near $1$. Therefore
\begin{eqnarray*}
\psi_\chi(1+it) & = & - \frac{\varphi(q)}{2 \pi q}\frac{\log(1+it)}{t} + \sum_{n = 0}^N (-1)^{n} C_{1,\b\chi}(n) \frac{t^n}{(1+it)^{n+1}}  + \sum_{n = 0}^N C_{1,\chi}(n) t^n + O(t^{N+1}).
\end{eqnarray*}
We now use the following series expansion near $0$:
\begin{eqnarray*}
\frac{1}{(1+it)^{n+1}} & = & \sum_{a= 0}^\infty {a+n \choose n} (-i)^a t^a \\
    \frac{\log(1+it)}{t} & = & \sum_{n=1}^\infty (-1)^{n+1} \frac{i^n}{n} t^{n-1},
\end{eqnarray*}
and switching the sums, we obtain the desired result by uniqueness of the expansion.

\bigskip

\section{Proof of Theorem \ref{main-th}} \label{sec-proof-th1}

First, note that we can use the same tools as in \cite{DH24}, especially the differential operator $\cal D:C^\infty(\R)\to C^\infty(\R), \cal D[ \varphi](x) = x\varphi'(x)$, \cite[Lemma 2.6]{DH24}, based on  \cite[Prop. 2.6(4)]{AIK14}.

Here, we identify in a natural way a specific non central Stirling number, see \cite{Kou82}, which simplifies the exposition in \cite{DH24}.

\begin{lemm}\label{compo-exp}
Let $\varphi\in C^{\infty}(\R)$ and $\phi(x)=e^{x/2}\varphi(e^x)$ for any $x\in\R$. Then, for all $n\geq 0$,
\begin{eqnarray*}
\phi^{(n)}(x) & = & \sum_{k=0}^n S_{-1/2}(n,k) e^{(k+1/2)x}\varphi^{(k)}(e^x).
\end{eqnarray*}
\end{lemm}
\begin{proof}
    By induction, we have
    \begin{eqnarray*}
    \phi^{(n)}(x) & = & \sum_{k=0}^n T(n,k) e^{(k+1/2)x}\varphi^{(k)}(e^x),
    \end{eqnarray*}
    where $(T(n,k))_{0\leq k\leq n}$ is defined by $T(n+1,k)=T(n,k-1)+(k+1/2)T(n,k)$, and $T(n,0)=1/2^n$ if $n\geq0$, $T(0,k)=0$ if $k\geq1$. This number is known to be the non-central Stirlling number $T(n,k)=S_{-1/2}(n,k)$, as defined in \cite{Kou82} Section 2.1.
\end{proof}

Writing
$\scr R_\chi(x) = e^{x/2} A_\chi(e^{x})$,
using the previous Lemma \ref{compo-exp},
and differentiating the integral expression of $B_{\chi}$, we obtain
\begin{eqnarray*}
\scr R_\chi^{(N)}(0)
= \frac{1}{2}\int_{-\infty}^\infty (it)^N e^{it\cdot 0}\left|L\left(\frac{1}{2}+it, \b \chi\right)\right|^{2} \frac{dt}{\cosh(\pi t)} 
= \sum_{n=0}^N S_{-1/2}(N,n) A_\chi^{(n)}(1).
\end{eqnarray*}
Hence, from Lemma \ref{fund-lemma}, i.e. $A_\chi(z) = -\chi(-1) i\pi\ \psi_\chi(z)$, and Theorem \ref{deriv-psi}, i.e.
\begin{eqnarray*}
        \psi_\chi^{(k)}(1) 
        & = & \frac{\varphi(q)}{2 \pi i q} \frac{(-1)^{k}}{k+1}k! + \scr B_{\chi}(k) k! + (-1)^k k!\sum_{j=0}^k {k \choose j} \scr B_{\b\chi}(j),
\end{eqnarray*}
where 
\begin{eqnarray}
    \scr B_\chi(k)
         & = & \frac{(2\pi i q)^k}{\tau(\chi) (k+1)(k+1)!} 
            \sum_{1\leq a,b\leq q} B_{k+1}(a/q)B_{k+1}(b/q)\ \chi(ab) \xi^{ab} \,
\end{eqnarray}
we finally obtain:
\begin{eqnarray*}
\int_{-\infty}^\infty t^N\left|L\left(\frac{1}{2}+it, \b \chi\right)\right|^{2} \frac{dt}{\cosh(\pi t)} & = & 
-\chi(-1) i\pi \frac{2}{i^{N}} \left( \frac{\varphi(q)}{2 \pi i q} \sum_{k=0}^N S_{-1/2}(N,k) \frac{(-1)^{k}}{k+1} k! \right.\\
     &   & \quad + \frac{1}{\tau(\chi)} 
            \sum_{1\leq a,b\leq q-1} \chi(ab)\ \xi^{ab} \cal P_{N}(a,b) \\
    &    & \left. \quad + \frac{1}{\tau(\b\chi)} 
            \sum_{1\leq a,b\leq q-1} \b\chi(ab)\ \xi^{ab} \cal R_{N}(a,b) \right),
\end{eqnarray*}
where
\begin{eqnarray*}
    \cal P_{\chi,N}(a,b) & = & \sum_{k=0}^N  S_{-1/2}(N,k) \frac{(2\pi i q)^k}{(k+1)^2} 
            B_{k+1}(a/q)B_{k+1}(b/q)\\
    \cal R_{\chi,N}(a,b) & = & \sum_{k=0}^N S_{-1/2}(K,k) (-1)^k k! \sum_{j=0}^k {k \choose j} \frac{q^j (2\pi i)^j}{(j+1)(j+1)!} 
            B_{j+1}(a/q)B_{j+1}(b/q) \\
            & = & \sum_{j=0}^N \frac{1}{j!} \sum_{k=j}^N S_{-1/2}(K,k) (-1)^k k! {k \choose j} \frac{(2\pi i q)^j}{(j+1)^2} 
            B_{j+1}(a/q)B_{j+1}(b/q).
\end{eqnarray*}

We now simplify the previous expressions.
\begin{lemm}
We have for $N\geq0$ and $0\leq j\leq n$:
\begin{eqnarray*}
\sum_{k=0}^N S_{-1/2}(N,k) \frac{(-1)^{k}}{k+1} k! & = & (2^{1-N}- 1) B_N. 
\end{eqnarray*}
\end{lemm}
\begin{proof}
We use Lemma 3.4 in \cite{DH24}, i.e.
\begin{eqnarray*}
\sum_{n=0}^N {N\choose n} 2^{n}\sum_{k=0}^n S(n,k) k! \frac{(-1)^{k}}{k+1} & = & (2-2^N)B_N, 
\end{eqnarray*}
and the expression (2.5) in \cite{Kou82}:\  $\d S_{-1/2}(N,k) = \frac{1}{2^N} \sum_{n=k}^N {N\choose n} 2^{n} S(n,k)$.
\end{proof}

The exponential generating function of $S_{-1/2}(n,k)$ is given by, see \cite[Sec. 2.2]{Kou82},
\begin{eqnarray} \label{gen-stirling}
    \sum_{n\geq k} S_{-1/2}(n,k) \frac{u^n}{n!} & = & \frac{e^{u/2}}{k!} (e^u-1)^k.
\end{eqnarray}
We now simplify $\cal R_{\chi,N}(a,b)$.
\begin{lemm}
    We have 
\begin{eqnarray*}
    S_{-1/2}'(N,j) := \frac{1}{j!}\sum_{k=j}^N S_{-1/2}(N,k) (-1)^k k! {k \choose j} = (-1)^N S_{-1/2}(N,j).
\end{eqnarray*}
\end{lemm}
\begin{proof}
We compute the corresponding exponential generating function:
\begin{eqnarray*}
    \sum_{n\geq j} (-1)^n S_{-1/2}'(n,j) \frac{u^n}{n!} & = & \sum_{n\geq j} \frac{(-1)^n}{j!}\sum_{k=j}^n S_{-1/2}(n,k) (-1)^k k! {k \choose j}  \frac{u^n}{n!} \\
        & = & \frac{1}{j!} \sum_{k\geq j}  {k \choose j} (-1)^k k! \sum_{n\geq k} S_{-1/2}(n,k) \frac{(-u)^n}{n!} \\
        & = & \frac{1}{j!} e^{-u/2} \sum_{k\geq j} {k \choose j} (-1)^k  (e^{-u}-1)^k \\
        & = & \frac{1}{j!} e^{-u/2} (1-e^{-u})^j \sum_{k\geq j} {k \choose j} (1-e^{-u})^{k-j} \\
        & = & \frac{1}{j!} e^{-u/2} \frac{(1-e^{-u})^j}{(e^{-u})^{j+1}}\ 
     = \ \frac{e^{u/2}}{j!} (e^u-1)^j,
\end{eqnarray*}
which gives the conclusion.
\end{proof}

Therefore
\begin{eqnarray*}
    \cal R_{\chi,N}(a,b) 
            & = & (-1)^N \sum_{j=0}^N S_{-1/2}(N,j) \frac{(2\pi i q)^j}{(j+1)^2} 
            B_{j+1}(a/q)B_{j+1}(b/q).
\end{eqnarray*}

Notice, using $\chi(-1)\tau(\b\chi)= \b{\tau(\chi)}$ and $|\tau(\chi)|^2=q$
\begin{eqnarray*}
\chi(-1)\left( \frac{\chi(ab)}{\tau(\chi)}   + 
              (-1)^N \frac{\b\chi(ab)}{\tau(\b\chi)} \right) & = & 
              \frac{\chi(ab)}{\b{\tau(\b\chi)}} + (-1)^N \frac{\b\chi(ab)}{\b{\tau(\chi)}}  \\
              & = & \frac{1}{q} \left( \chi(ab) \tau(\b\chi) + (-1)^N \b\chi(ab) \tau(\chi) \right).
\end{eqnarray*} 

We now write $\d \left(\frac{1}{2^{N-1}} -1\right)B_N = B_N\left(\frac{1}{2}\right)$ in the first term, see e.g. \cite[Cor. 9.1.5 p.5-6]{Coh07}.

Multiplying by $2\pi i$ and $-1/(i^N)=-(-i)^N$, and then factorizing the expression, yields the form given in Theorem \ref{main-th} (replacing $\b \chi$ by $\chi$).

\bigskip

\section{The shifted Euler-Maclaurin summation formula} \label{sec-mac}

The goal of this section is to establish the variant of the Euler-Maclaurin formula that we will use. This variant is inspired by~\cite{Zag06}.

For all \(\varphi \in \left(-\tfrac{\pi}{2}, \tfrac{\pi}{2}\right)\), we set
\[
\Cone_\varphi^* = \left\{ z \in \C^*  : |\arg(z)| < \varphi \right\} \quad\text{and}\quad
\Cone_\varphi = \Cone_\varphi^* \cup \{0\}.
\]
A function \(f\) is said to be rapidly decreasing if it is analytic on \(\re_{>0}\) and satisfies
\[
\forall (m,n) \in \mathbb{Z}_{\geq 0}^2 \quad \lim_{\substack{z \to \infty \\ z \in \re_{>0}}} z^m f^{(n)}(z) = 0.
\]

Finally, we say that a function $f$ defined on \(\re_{>0} \cup \{0\}\) is \(\Cone\)-analytic if:
\begin{itemize}
\ib $f$ is rapidly decreasing;
\ii $f$ admits an absolutely convergent power series expansion 
\[
f(z) = \sum_{n=0}^{\infty} \CDSE{f}{n} z^n
\]
\end{itemize}
near \(0\) (and hence is analytic on an open set \(\Cone_\varphi\cup B(0,\rho)\) for all \(\varphi \in \left(-\tfrac{\pi}{2}, \tfrac{\pi}{2}\right)\) and some \(\rho>0\)).

\begin{theo}\label{main-shift}
Let \(f \colon \re_{>0} \to \mathbb{C}\). Assume that there exists a constant \(\CDSE{f}{-1}\) such that the function \(z \mapsto f(z)- \tfrac{\CDSE{f}{-1}}{z}\) is \(\Cone\)-analytic. Then \(f\) admits 
 a Laurent expansion with coefficients \(\left( \CDSE{f}{n} \right)_{n \geq -1}\).

Let \(h\colon\mathbb{R}_{>0} \to \re_{>0}\). Suppose that \(h(t) \to 0\) as \(t \to 0\), and that there exist \(\varphi \in \left(-\tfrac{\pi}{2}, \tfrac{\pi}{2}\right)\) and \(t_0 \in\mathbb{R}_{>0} \) such that \(h((0, t_0)) \subset \Cone_\varphi\).

Then, for any \(\alpha \in (0,1)\), we have
\begin{multline*}
\sum_{m=0}^{\infty} f\left((m+\alpha) h(t)\right)
= -\CDSE{f}{-1} \frac{\log h(t)}{h(t)}
+ \left(
\CDSE{f}{-1} K(\alpha)
+ \int_0^{\infty} \left(f(x) - \CDSE{f}{-1} \frac{e^{-x}}{x} \right) dx
\right) \frac{1}{h(t)} \\
- \sum_{n=0}^{N} \frac{B_{n+1}(\alpha)}{n+1} \CDSE{f}{n} h(t)^n
+ O\left(h(t)^N\right)
\end{multline*}
where \(\d K(\alpha) = -\gamma - \tfrac{\Gamma'}{\Gamma}(\alpha)\).
\end{theo}

The rest of this section is devoted to proving this theorem, starting with two auxiliary lemmas. The arguments are adapted from those presented in \cite{Bak22}.

\begin{lem}\label{Thm:ShiftEM}
Let \(f\) be a \(\Cone\)-analytic function. Then
\begin{eqnarray*}
\sum_{m=0}^{\infty} f\left((m+\alpha)z\right)
& = & \left(\int_0^{\infty} f\right) \frac{1}{z}
- \sum_{n=0}^{N} \frac{B_{n+1}(\alpha)}{n+1} \CDSE{f}{n} z^n
- R_N(z)
\end{eqnarray*}
where
\begin{eqnarray*}
R_N(z) & = & (-z)^{N+1} \int_0^{\infty} \frac{B_{N+1}\left(\{x - \alpha\}\right)}{(N+1)!} f^{(N+1)}(xz) dx
\end{eqnarray*}
for all \(z \in \re_{>0}\). Moreover, for any \(\varphi \in \left(-\tfrac{\pi}{2}, \tfrac{\pi}{2}\right)\), as \(z \to 0\) in \(\Cone_\varphi\), we have \(R_N(z) = O(z^N)\).
\end{lem}
\begin{proof}
By integration by parts,
\begin{multline*}
\int_0^\alpha \frac{B_n(x - \alpha + 1)}{n!} f^{(n)}(x) dx
= \frac{B_{n+1}(1)}{(n+1)!} f^{(n)}(\alpha) - \frac{B_{n+1}(1 - \alpha)}{(n+1)!} f^{(n)}(0) \\
- \int_0^\alpha \frac{B_{n+1}(x - \alpha + 1)}{(n+1)!} f^{(n+1)}(x) dx
\end{multline*}
and
\begin{multline*}
\int_\alpha^1 \frac{B_n(x - \alpha)}{n!} f^{(n)}(x) dx
= \frac{B_{n+1}(1 - \alpha)}{(n+1)!} f^{(n)}(1) - \frac{B_{n+1}(0)}{(n+1)!} f^{(n)}(\alpha) \\
- \int_\alpha^1 \frac{B_{n+1}(x - \alpha)}{(n+1)!} f^{(n+1)}(x) dx
\end{multline*}
since \(B'_{n+1}(x) = (n+1) B_n(x)\). For \(n \geq 1\), we have \(B_{n+1}(0) = B_{n+1}(1)\), \(B_1(1) = -B_1(0) = 1/2\), and \(B_{n+1}(1 - \alpha) = (-1)^{n+1} B_{n+1}(\alpha)\). Summing these two equalities and then summing the resulting identity over \(n\), we get
\[
\int_0^1 f - f(\alpha)
= -\sum_{n=0}^N \frac{B_{n+1}(\alpha)}{(n+1)!} \left(f^{(n)}(1) - f^{(n)}(0)\right)
+ (-1)^{N+1} \int_0^1 \frac{B_{N+1}(\{x - \alpha\})}{(N+1)!} f^{(N+1)}(x) dx.
\]
Applying this identity to the function \(x \mapsto f(x + m)\) and summing over \(m\), we obtain
\[
\sum_{m=0}^{\infty} f(\alpha + m)
= \int_0^{\infty} f
- \sum_{n=0}^N \frac{B_{n+1}(\alpha)}{n+1} \frac{f^{(n)}(0)}{n!}
+ (-1)^N \int_0^{\infty} \frac{B_{N+1}(\{x - \alpha\})}{(N+1)!} f^{(N+1)}(x) dx.
\]
We then apply this identity to the function \(x \mapsto f(xz)\) for \(z \in \re_{>0}\). Noting that
\[
z \int_0^{\infty} f(xz) dx = \int_0^{\infty} f
\]
(for instance by integrating over a contour supported on the real axis, the line through \(0\) with direction \(z\), bounded by two arcs centered at \(0\), or via holomorphic extension from the identity valid for real \(z\)), we obtain
\[
\sum_{m=0}^{\infty} f\left((\alpha + m)z\right)
= \frac{1}{z} \int_0^{\infty} f
- \sum_{n=0}^N \frac{B_{n+1}(\alpha)}{n+1} \frac{f^{(n)}(0)}{n!} z^n
+ \widetilde{R}_N(z)
\]
where
\[
\widetilde{R}_N(z)
= (-1)^N z^{N+1} \int_0^{\infty} \frac{B_{N+1}(\{x - \alpha\})}{(N+1)!} f^{(N+1)}(xz) dx.
\]

The integral is \(O(1/|z|)\) near \(0\). Indeed, suppose \(|z| \geq 1\). The integral over \([0, 1/|z|]\) is \(O(1/|z|)\), since its integrand is uniformly bounded by \(\| B_{N+1} f^{(N+1)} \|_\infty\), the norm being taken over \([0,1]\); the integral over \([1/|z|, \infty)\) is also \(O(1/|z|)\) due to the rapid decay of \(f\), which implies the existence of a constant \(A > 0\) such that \(|f^{(N+1)}(\xi)| \leq A / |\xi|^2\) for \(|\xi| \geq 1\). We deduce that \(|\widetilde{R}_N(z)| = O(|z|^N)\).
\end{proof}
\begin{rem}
Lemma~\ref{Thm:ShiftEM} implies, under the same assumptions,
\begin{eqnarray*}
\sum_{m=0}^{\infty} f\left((m+\alpha)z\right)
& = & \left(\int_0^{\infty} f\right) \frac{1}{z}
- \sum_{n=0}^N \frac{B_{n+1}(\alpha)}{n+1} \CDSE{f}{n} z^n
+ O\left(z^{N+1}\right).
\end{eqnarray*}
\end{rem}
\smallskip

\begin{cor}\label{shift-exp}
Let \(\varphi \in \left(-\tfrac{\pi}{2}, \tfrac{\pi}{2}\right)\). As \(w \to 0\) in \(\Cone_\varphi^*\), we have
\begin{eqnarray*}
\sum_{m=0}^{\infty} \frac{e^{-(m+\alpha)w}}{(m+\alpha)w}
& = & -\frac{\log(w)}{w}
+ \frac{K(\alpha)}{w}
+ \sum_{m=0}^N \frac{B_{m+1}(\alpha)}{m+1} \frac{(-1)^m}{(m+1)!} w^m
+ O\left(w^{N+1}\right)
\end{eqnarray*}
where \(K(\alpha) = -\gamma - \frac{\Gamma'}{\Gamma}(\alpha)\).
\end{cor}

\begin{proof}
Let \(w \in \re_{>0}\) and \(\varphi \in \left(-\tfrac{\pi}{2}, \tfrac{\pi}{2}\right)\) such that \(w \in \Cone_\varphi^*\). We apply Lemma~\ref{Thm:ShiftEM} to write
\[
\sum_{m=0}^{\infty} e^{-(m+\alpha)\xi}
= \frac{1}{\xi}
- \sum_{n=0}^{N+1} \frac{(-1)^n B_n(\alpha)}{(n+1)!} \xi^n
- R_{N+1}(\xi)
\]
for all \(\xi \in \Cone_\varphi^*\). The bound \(B_n(\alpha)=O\left(n!/(2\pi)^n\right)\) (see~\cite{Leh40}) implies that
\[
R_{N+1}(\xi)=\sum_{n=N+2}^{\infty} \frac{(-1)^n B_n(\alpha)}{(n+1)!} \xi^n
\]
and hence, \(R_{N+1}\) has a primitive on a disc centered a \(0\). Let \(\widetilde{R}_{N+1}\) denote such a primitive  vanishing at \(0\).

We get
\[
\sum_{m=0}^{\infty} \frac{e^{-(m+\alpha)w}}{m+\alpha}
= K(\alpha)
+ \log(w)
- \sum_{n=0}^{N+1} \frac{(-1)^n B_n(\alpha)}{(n+1)(n+1)!}w^{n+1}
- \widetilde{R}_{N+1}(w).
\]

To evaluate \(K(\alpha)\), we observe that
\begin{align*}
\sum_{m=0}^{\infty} \frac{e^{-(m+\alpha)\epsilon}}{m+\alpha}
&= \sum_{m=0}^{\infty} \int_\epsilon^{\infty} e^{-(m+\alpha)x} dx
= \int_\epsilon^{\infty} \frac{e^{(1-\alpha)x}}{e^x - 1} dx \\
&= e^{-\alpha \epsilon} \log(e^\epsilon - 1)
+ \alpha \int_\epsilon^{\infty} e^{-\alpha x} \log(e^x - 1) dx
\end{align*}
so that
\[
\sum_{m=0}^{\infty} \frac{e^{-(m+\alpha)\epsilon}}{m+\alpha}
+ \log(\epsilon)
= \alpha \int_0^{1} (1-u)^{1-\alpha}\log(u) du + o(1)= -\gamma - \frac{\Gamma'}{\Gamma}(\alpha) +o(1)
\]
using~\cite[4.253.1,\ 8.384.1,\ 8.366.1,\ 8.365.1]{GR07}.


Finally,
\[
|\widetilde{R}_{N+1}(w)| \leq \max_{z \in [\epsilon, w]} |R_{N+1}(z)| \cdot |w - \epsilon|
= O(|w|^{N+1}).
\]
\end{proof}

\begin{proof}[\proofname{} of Theorem~\ref{main-shift}]
Define \(g(z) = f(z) - \CDSE{f}{-1} \frac{e^{-z}}{z}\). Then \(g\) is \(\Cone\)-analytic and
\[
\CDSE{g}{n} = \CDSE{f}{n} + \frac{(-1)^n}{(n+1)!} \CDSE{f}{-1}
\]
for all integers \(n \geq 0\).
Lemma~\ref{Thm:ShiftEM} then gives
\begin{multline*}
\sum_{m=0}^{\infty} g\left((m+\alpha) h(t)\right)
= \frac{1}{h(t)} \int_0^{\infty} \left(f(t) - \CDSE{f}{-1} \frac{e^{-t}}{t}\right) dt \\
- \sum_{n=0}^N \frac{B_{n+1}(\alpha)}{n+1}
\left( \CDSE{f}{n} + \frac{(-1)^n}{(n+1)!} \CDSE{f}{-1} \right) h(t)^n
+ O\left(h(t)^N\right).
\end{multline*}
The result then follows from Corollary~\ref{shift-exp}.
\end{proof}

\bigskip

\section{Examples and numerical experiments} \label{sec-num}

\subsection{A few remarks and examples}
We now set 
\begin{eqnarray*}
m_N(\chi) & = & \int_{-\infty}^\infty t^{N}\left|L\left(\frac{1}{2}+it, \chi\right)\right|^{2} \frac{dt}{\cosh(\pi t)}.
\end{eqnarray*}

For any character, we have from the change of variable $t\to-t$:
$m_{2N+1}(\chi)=-m_{2N+1}(\b\chi)$.
Hence, if $\chi$ is real, the odd moments are $0$. If $\chi$ is not real, the odd moments may be positive or negative; an example with negative odd moments is given below.

For any odd primitive quadratic character $\chi$, Theorem \ref{main-th} gives:
\begin{eqnarray*}
m_{0}(\chi) & = & \frac{\varphi(q)}{q}   - \frac{4\pi i \tau(\chi)}{q} 
            \sum_{1\leq a,b\leq q} \left(\frac{a}{q}-\frac{1}{2}\right)\left(\frac{b}{q}-\frac{1}{2}\right) \chi(ab) \xi^{ab} .
\end{eqnarray*}
Notice that $\b{\tau(\chi)}=\chi(-1)\tau(\b\chi)=-\tau(\chi)$, so $4\pi i\tau(\chi)\in\R$, and so is the sum (which can also be checked in an elementary way). 

Remarkably, for $\chi_4$,  the sum above is $0$, and $\d m_{0}(\chi_4)=\frac{1}{2}$. The $L-$function associated to this character has many interesting properties, see also \cite[p.458]{BPY01}.
Also note that 
$$
m_0(\chi_3)  = \frac{2}{3} - \frac{2\pi}{9\sqrt{3}} = 0.2636001412\ldots
$$
for the primitive odd character $\chi_3$ modulo $3$, which shows that the sign of the first term right after the Stirling number can be negative.

\subsection{The first moments of the first odd characters}

The characters $\chi_3$ and $\chi_4$ are real and there are two complex conjugate characters modulo $5$. Define $\chi_5(2)=-\chi_5(3)=-i$, $\chi_5(-1)=-1$. 
\\
\renewcommand{\arraystretch}{1.3}
\begin{center}
\begin{tabular}{|c|c|c|c|}
    \hline
    $k$ & $m_k(\chi_3)$ & $m_k(\chi_4)$ & $m_k(\chi_5)$ \\
    \hline
    0 & $0.263600\ldots$ & $0.5$        & $0.699061\ldots$ \\
    1 & $0$             & $0$             & $-0.045214\ldots$ \\
    2 & $0.101605\ldots$ & $0.178744\ldots$ & $0.238198\ldots$ \\
    3 & $0$             & $0$             & $-0.082899\ldots$ \\
    4 & $0.217067\ldots$ & $0.334811\ldots$ & $0.414305\ldots$ \\
    5 & $0$             & $0$             & $-0.338419\ldots$ \\
    6 & $1.091158\ldots$ & $1.413929\ldots$ & $1.616508\ldots$ \\
    7 & $0$             & $0$             & $-2.209431\ldots$ \\
    8 & $9.310390\ldots$ & $9.787352\ldots$ & $10.565897\ldots$ \\
    9 & $0$             & $0$             & $-19.075490\ldots$ \\
    10 & $114.676457$   & $95.077818\ldots$    & $103.068850\ldots$ \\
    \hline
\end{tabular}
\end{center}
\renewcommand{\arraystretch}{1}

\smallskip

\subsection{The moment $0$ and the central point}
It is tempting to study the ratio between $m_0(\chi)$ and $|L(1/2,\chi)|^2$, which is displayed in Figure \ref{fig:RatioMomentHalfL2} up to the $100$th primes.

Large values of this ratio do not appear to be particularly rare, and a more detailed investigation using PARI/GP reveals that increasingly exceptional behavior, relative to an average trend, can occur as the modulus grows.

\begin{figure}[H]
    \centering
    \includegraphics[scale = 0.5]{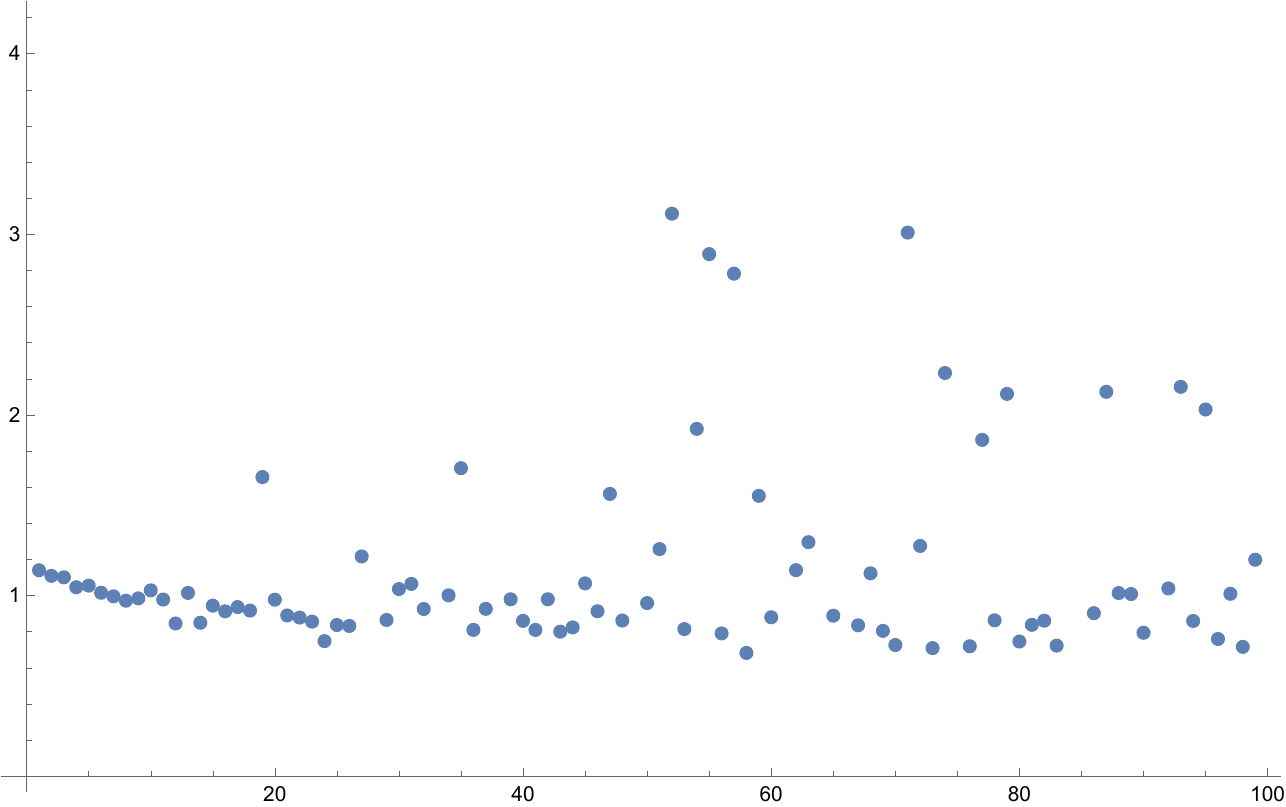}
    \caption{Graph of $\d n\longmapsto \frac{m_0(\chi_{p_n})}{|L(1/2,\chi_{p_n})|^2}$, where $\chi_{p_n}$ is a random character modulo the $n-$th prime $p_n$, $n\leq 100$.}
    \label{fig:RatioMomentHalfL2}
\end{figure}

\subsection{About the coefficients in the expansion $\psi_\chi(1+z)$}

We write the Taylor series of $\psi_{\chi}$ around $1$:
\[
\psi_{\chi}(1 + z) = \sum_{n = 0}^\infty \psi_{n,\chi} z^n, \quad \psi_{n,\chi}= \frac{\psi_\chi^{(n)}(1)}{n!},
\]
where the formula for $\psi_{n,\chi}$ is given by Theorem \ref{deriv-psi}. We now fix $q = 5$ and $\chi = \b \chi_5$ the non-real character defined by\  $\b \chi_5(2)=-\b \chi_5(3)=i$, $\b \chi_5(-1)=-1$.

First, the values  $\psi_n$ are surprisingly small (see Figure \ref{fig: SinGraph}). Actually they seem to decay sub-exponentially in $n$. The reason this is a surprise is that the values $ \scr B_\chi(n)$ in Theorem \ref{deriv-psi} seem to grow at least exponentially in $n$. A similar phenomena is also recorded in \cite[Th. 2 p.5712]{BC13b}. 

\begin{figure}[H]
     \centering
    \includegraphics[]{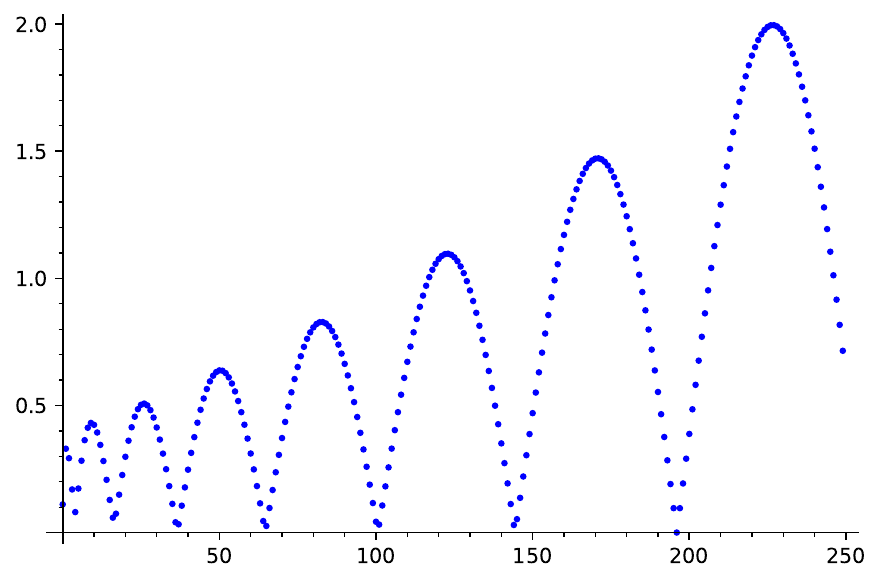}
    \caption{Graph of $n\longmapsto e^{\sqrt{\pi n}}|\psi_{n, \chi}|$ for $n \leq 250$.}
\label{fig: SinGraph}
\end{figure}

 Second, we observe that the values $\psi_{n, \chi}$ roughly lie on a line in the $(\re, \im)$-plane. This is reflected in Figure \ref{fig:Psinonline}. It seems to be natural to plot the ratio  $R_n := \im(\psi_{n,\chi})/\re(\psi_{n, \chi})$. The result is Figure \ref{fig:Psiratio}.
It is quite apparent from the picture that $R_n$ interpolates a smooth function $R_x$ with some poles. Also, the behavior around these poles seems to be random. 

As observed also from Figure \ref{fig:Psinonline} it seems that Figure \ref{fig:Psiratio}  hovers around a horizontal line (in this particular case it seems to be around $-\phi$, the golden ratio, but this is specific to $q = 5$).
We do not yet have an explanation for these curiosities.

\begin{figure}[H]
    \centering
    \includegraphics[]{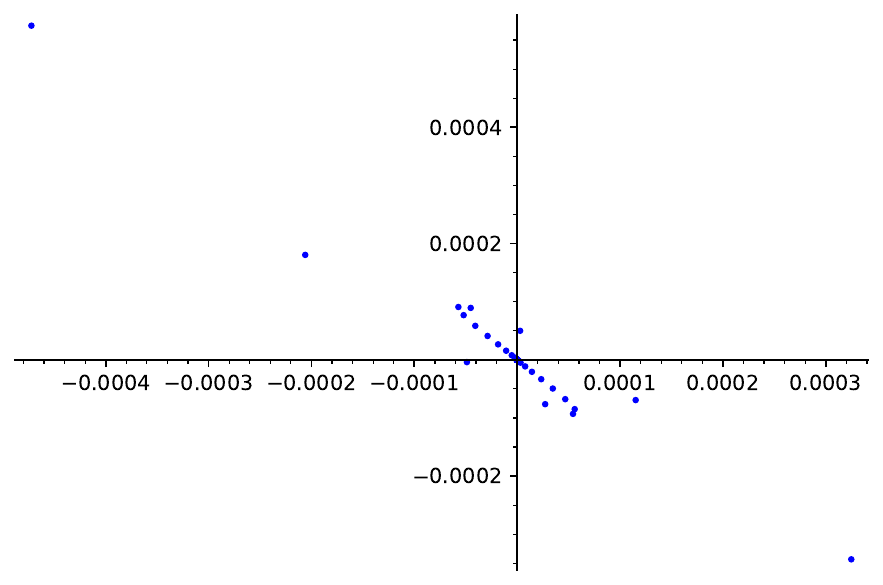}
    \caption{Points $\psi_{n,\chi}$ plotted in the complex plane for $12 \leq n \leq 250$.}
    \label{fig:Psinonline}

    \centering
    \includegraphics[]{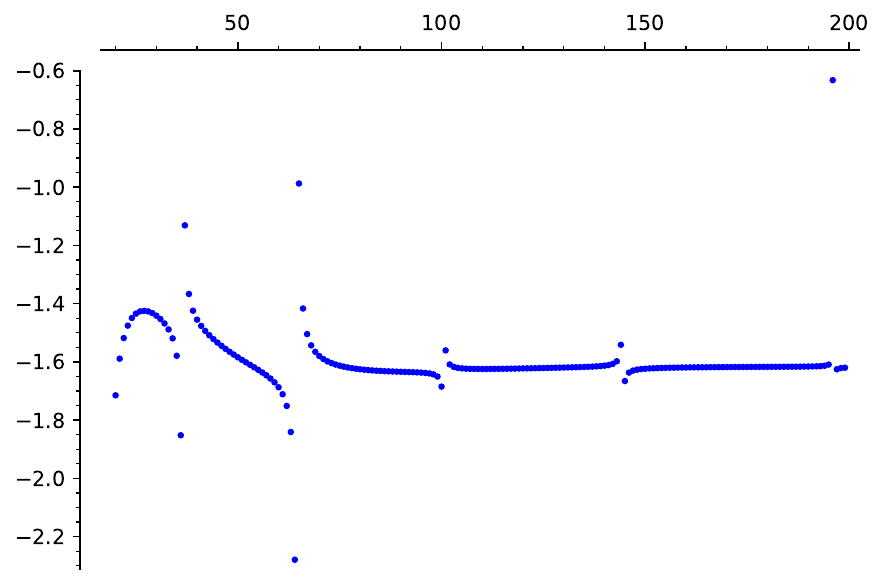}
    \caption{Graph of $n \longmapsto R_n := \im(\psi_{n, \chi})/\re(\psi_{n, \chi})$ for $20 \leq n \leq 200$.}
    \label{fig:Psiratio}
\end{figure}

\bigskip

\section*{Acknowledgement} 

The authors are very grateful to Christophe Delaunay for initiating this project and many insightful discussions, and to Olivier Ramar\'e for his very interesting suggestions. They both anticipated possible generalizations of the formula in \cite{DH24}, independently. 

S.D. is in debt to Youness Lamzouri for discussing various generalizations, 
pointing out relations with Fekete polynomials, and many other aspects including higher moments of $L-$functions, and finally for communicating the reference \cite{Ram80}.
He also thanks Lucas Benigni, Erik Carlsson, Francesco Cellarosi, Thierry Daud\'e, Hélène Guérin, Erwan Hillion, Joseph Najnudel, Dan Romik, Zachary Selk and Thomas Simon for stimulating discussions and insightful connections.

E.R. is partially funded by the ANR-23-CE40-0006-01 Gaec project

\bigskip


\end{document}